\documentclass[11pt,a4paper,english]{amsart}
\usepackage{amsaddr}
\usepackage{babel}
\usepackage[latin1]{inputenc}
\usepackage{enumitem}
\usepackage{amsmath, amscd}
\usepackage{amsfonts}
\usepackage{amssymb}
\usepackage{amsthm}
\usepackage{graphicx}
\usepackage{multicol}
\usepackage{color}
\usepackage{microtype}
\usepackage{calc}
\usepackage[margin=75pt]{geometry}
\usepackage[hyperindex]{hyperref}

\newtheorem{theo}{Theorem}
\newtheorem{prop}[theo]{Proposition}

\newtheorem{defin}[theo]{Definition}
\newtheorem{lemma}[theo]{Lemma}
\newtheorem{main}{Main theorem}

\newtheorem{rem}[theo]{Remark}

\newtheorem{ass}{Assumption}

\numberwithin{equation}{section}

\newcommand{\R}{\mathbb{R}}
\newcommand{\N}{\mathbb{N}}
\newcommand{\T}{\mathbb{T}}

\newcommand{\Oo}{\mathcal{O}}

\newcommand{\Z}{\mathbb{Z}}

\newcommand{\id}{{\rm id}\,}
\newcommand{\cl}{{\rm cl}\,}
\newcommand{\rk}{{\rm rank}\,}
\newcommand{\interior}{{\rm int}\,}
\newcommand{\Lie}{{\rm Lie}\,}
\newcommand{\Dom}{{\rm Dom}\,}
\newcommand{\codim}{{\rm codim}\,}
\newcommand{\Span}{{\rm Span}\,}

\title{Generic transitivity for couples of Hamiltonians}

\begin{document}

\begin{abstract}
We study orbits and reachable sets of generic couples of Hamiltonians $H_1, H_2$ on a symplectic manifold $N$. We prove that, $C^k$-generically for $k$ large enough, orbits coincide with the whole of $N$, and that the same is true for reachable sets when $N$ is compact.

Our results are stated in terms of a strong form of genericity which makes use of the notion of rectifiable subsets of positive codimension in Banach or Frechet spaces.

\end{abstract}

\author{Vito Mandorino \textsuperscript{*}}
\address{
\emph{\texttt{vito.mandorino@math.u-psud.fr}}}
\thanks{* Université Paris-Dauphine \& Université Paris-Sud
\smallskip}


\maketitle

\tableofcontents

All considered manifolds are separable, finite dimensional, 
 smooth, connected and without boundary, unless otherwise stated.

\section{Introduction}

In this paper we consider the switched system associated to a generic couple of Hamiltonians $H_1,H_2$ on a symplectic manifold $N$. 
 Our focus is to prove that orbits and reachable sets of such a system are generically  the whole of $N$.
  Our work is in the same spirit of a paper of Lobry \cite{Lob72} who proved that, for a $C^k$-generic couple of vector fields on a manifold $M$ and for $k$ large enough, orbits are the whole of $M$.

 We will come back later to our notion of genericity, which is given in terms of rectifiable subsets of positive codimension in a Banach or Frechet space and which is stronger than the usual one based on the Baire Category Theorem.
 Let us now quickly recall the definitions of orbit and reachable set.
  \medbreak
  
 Let us denote by $\{\phi^t_{H_1}\}_{t}$ and $\{\phi^t_{H_2}\}_{t}$ the   Hamiltonian (local) flows   of $H_1$ and $H_2$ respectively. 
   We assume that $H_1$ and $H_2$ are at least $C^2$.\ 
The \emph{orbit} $\Oo_{H_1,H_2}(z)\subseteq N$ of a point $z\in N$ through the switched system associated to $H_1$ and $H_2$ is obtained by applying to $z$ the group\footnote{This should be more properly called a pseudogroup, since the flows may not be complete.} generated by the two flows. More explicitly, 
 \[
 \Oo_{H_1,H_2}(z)=\left\{ \phi^{t_n}_{H_{i_n}}\circ\dots\circ\phi^{t_1}_{H_{i_1}}(z) :i_1,\dots,i_n\in\{1,2\},t_1,\dots,t_n\in\R,n\ge 1 \right\}.
 \]
 The \emph{reachable set} of $z$, denoted by $\Oo^+_{H_1,H_2}(z)$, represents instead the ``future'' of $z$, and is obtained by applying to $z$ the (pseudo-)semigroup generated by $\{\phi^t_{H_1}\}_{t>0}$ and $\{\phi^t_{H_2}\}_{t>0}$. More explicitly,
 \[
  \Oo^+_{H_1,H_2}(z)=\left\{ \phi^{t_n}_{H_{i_n}}\circ\dots\circ\phi^{t_1}_{H_{i_1}}(z) :i_1,\dots,i_n\in\{1,2\},t_1,\dots,t_n>0,n\ge 1 \right\}.
 \]
 By taking $t_1,\dots,t_n<0$ in the expression above one gets the analogous definition for the negative reachable set $\Oo^-_{H_1,H_2}(z)$, i.e.\ the ``past'' of $z$.

 When interested in orbits, we say that we are in the \emph{time-unoriented case}. We say that we are in the \emph{time-oriented case} when we are interested in reachable sets.
 \medbreak
 
 The main result of this paper is that for a $C^k$-generic couple $H_1,H_2$ and for $k$ large enough orbits are the whole of $N$, and the same is true for reachable sets provided that the manifold $N$ is compact. In fact, we prove something more accurate than $C^k$-genericity of couples: in a first respect, we adopt the notion of rectifiable set of positive codimension in Banach or Frechet spaces, which is a stronger notion\footnote{It is also stronger than some other notions of translational invariant ``smallness'' in infinite-dimensional spaces, such  as prevalence or Aronszajn-nullity, see \cite{Ber10}.} than genericity in the sense of Baire: indeed,  such a rectifiable set  is always Baire-meager whereas the viceversa is not always true  (see \cite{Ber10} for a detailed study; the basic facts and definitions are recalled in Section \ref{sect:thom}). In a second respect, we make perturbations just in $H_2$ leaving $H_1$ fixed (apart from a  small subset of highly degenerate $H_1$, see Assumption \ref{ass:H_1}  
  below).

Let us set $\dim N=2d$. Recall that, for each $k\in N$, the space $C^k(N)$ of $C^k$-real functions on $N$ is a Banach space when $N$ is compact, 
 and a Frechet space otherwise. Our main results are as follows:
 
 \begin{main}[Time-unoriented case]
 Let $H_1\in C^{4d+1}(N)$ satisfy the non-degeneracy assumption \ref{ass:H_1} below. Let $k\ge 4d$. Then, the set  
\[
\left\{ H_2\in C^{k}(N): \Oo_{H_1,H_2}(z)=N\ \forall\,z\in N \right\}
\]
has rectifiable complementary of codimension $\ge 1$ in $C^{k}(N)$. In particular, it is $C^k$-generic.
\end{main}

\begin{main}[Time-oriented case]
Assume that $N$ is compact. Let $H_1$ and $k$ be as before. Then, 
 the  set
\[
\left\{ H_2\in C^{k}(N): \Oo^+_{H_1,H_2}(z)=N\ \forall\,z\in N \right\}
\]
has rectifiable complementary of codimension $\ge 1$ in $C^{k}(N)$, and in particular it is $C^k$-generic.
\end{main}

 The two results above are proved in Section \ref{sect:autonomous} (Theorems \ref{theo:tu-aut} and \ref{theo:to-aut}). They extend quite naturally to time-dependent Hamiltonians as well, this is the content of the last Section \ref{sect:per}. 
 \medbreak
 
 The proofs make fundamentally use of three ingredients, namely the Rashevski-Chow Theorem, the Thom transversality Theorem and the Hamiltonian flow-box Theorem. The exact statements serving our purposes will be given in Section \ref{sect:preliminaries}.
  
  The Rashevski-Chow Theorem, as it is well-known, makes a link between the Lie algebra spanned by two vector fields at a point and the  orbit or reachable set of that point. The part concerning reachable sets is more precisely called Krener Theorem. In Section \ref{sect:chow} we recall some local and global versions of these theorems. In the global time-oriented version  we will make the additional assumption that the flows of the considered vector fields have non-wandering dynamics. Under the assumptions of Krener Theorem, this is a sufficient condition for  concluding that reachable sets are equal to the whole manifold.  On the other hand, every Hamiltonian flow on a compact symplectic manifold is non-wandering by Poincaré recurrence Theorem; this is the reason for the compactness assumption on $N$ in the statement of the Main Theorem 2.

 Concerning  the Thom Transversality Theorem, we use a refined version of the classical result which has been recently proved in \cite{BerMan12}. Such a version yields that the set of maps whose jet is transverse to a submanifold in a jet space not only is generic, but its complementary is rectifiable of positive codimension, which, as already mentioned, is a  stronger information. We recall the result in Section \ref{sect:thom} along with the notion of rectifiable set of positive codimension in Banach (or Frechet) spaces. 
 
 The last ingredient is the Hamiltonian flow-box Theorem, a normal form result which makes computation of iterated Lie (or Poisson) brackets handleable. The exact statement is recalled in Section \ref{sect:flowbox}.
  \medbreak
  
We finish the introduction by stating the non-degeneracy assumption required to $H_1$ in the main results above. Let us first give a definition.

\begin{defin}
\label{defin:codim}
Let $S\subseteq M$ be a subset of the manifold $M$. We say that $S$ has codimension $c$ in $M$, and we write $\codim_M S=c$, if
\[
S\subseteq \bigcup_{l\in\N}S^{l}
\]
for a countable family $\{S^{l}\}_l$ of $C^1$-submanifolds of codimension $\ge c$ in $M$.
\end{defin}

Note that, according to the above definition, a subset $S$ of codimension $c$ is also of codimension $c'$ for every $c'\le c$. In fact, in this paper we are interested in estimate the size of certain subsets, and only lower bounds (rather than sharp values) for codimension will matter.

\begin{ass}
\label{ass:H_1}
The Hamiltonian $H_1$ is continuously differentiable and the subset $N_0\subseteq N$ defined by
\[
N_0=\{ z\in N:d_z H_1=0 \}
\]
is contained in a countable union $\cup_{l\in\N} N^{l}$ of ${C^2}$-submanifolds of $N$ of codimension greater or equal than
 $\frac{\dim N}2 +1$.
 Given such a family\footnote{We will always assume that, given $H_1$ as above, the family $\{N^{l}\}_l$ has been chosen once for all. Such a choice will never play any role.} $\{N^l\}_l$ we define the tangent space $TN_0$ as
\[
TN_0=\bigcup_{l\in \N} TN^{l}.
\]
Note that each $TN^l$ is a submanifold of class $C^1$. According to Definition \ref{defin:codim} we have
\begin{align*}
\codim_N N_0&\ge \frac{\dim N}2 +1
\\
\codim_{TN} TN_0&\ge 2\,\codim_N N_0\ge \dim N+2.
\end{align*}
\end{ass}

The Assumption \ref{ass:H_1} is generically satisfied, by an immediate application of the classical Thom transversality Theorem.

\bigbreak

\textsc{Acknowledgements} - I would like to thank my thesis advisor Patrick Bernard for useful remarks and suggestions at different stages of this paper.

\section{Notation and preliminaries}
\label{sect:preliminaries}

\subsection{The Rashevski-Chow Theorem: time-unoriented and time-oriented versions}
\label{sect:chow}

We state here several versions of the Rashevski-Chow theorem for two (non-necessarily Hamiltonian) vector fields. For the proofs  we refer to the books \cite{Jur97,AgrSac04}. The proofs in these references are sometimes given for smooth vector fields, but they hold unchanged in the $C^k$ case, $k\ge 1$.

Let us first establish some notation. Given two $C^{k}$ vector fields $X_1,X_2$ ($k\ge 1$) on a manifold $M$, the orbit $\Oo_{X_1,X_2}(z)$ and the reachable set $\Oo^+_{X_1,X_2}(z)$ of a point $z\in M$ are defined as in the introduction, with the (local) flows of $X_1$ and $X_2$ playing the role of the Hamiltonian flows of $H_1$ and $H_2$ of the introduction.

We denote by $\Lie^{k}(X_1,X_2)$ the vector space  spanned by the vector fields
\[
X_1,\quad X_2,\quad[X_{i_0},[X_{i_1},[\dots,[X_{i_{m-1}},X_{i_{m}}],\dots]]],\qquad 1\le m\le k,\ i_0,i_1,\dots,i_m\in\{1,2\}.
\]
where $[\cdot,\cdot]$ denotes the usual bracket of vector fields. We also denote by $\Lie^{k}_1(X_1,X_2)$ the vector space obtained by bracketing just with $X_1$, i.e.~the span of the $k+2$ vector fields
\[
X_1,\quad X_2,\quad \underbrace{[X_{1},[X_{1},[\dots,[X_{1}}_\text{$m$ times},X_{2}]\dots]]],\qquad 1\le m\le k.
\]
We obviously have
\[
\Lie^{k}_1(X_1,X_2)\subseteq \Lie^k(X_1,X_2).
\]
Evaluating these vector spaces at a point $z\in M$ yields a vector subspace of $T_z M$: we will use the notation
\[
\Lie^{k}_1(X_1,X_2)(z),\quad \Lie^k(X_1,X_2)(z)
\]
to denote these vector subspaces.  

 For two Hamiltonians $H_1,H_2$ defined on a symplectic manifold $(N,\omega)$, let us recall the basic identity
\[
X_{\{H_1,H_2\}}=[X_{H_1},X_{H_2}]
\]
where $\{\cdot,\cdot\}$ is the usual Poisson bracket of functions and $X_{H_1}, X_{H_2}$ are the Hamiltonian vector fields of $H_1,H_2$ defined by the usual formula, valid for any function $H$,
\[
\iota_{X_H}\omega=dH.
\]
In the next sections we will tacitly use the following immediate consequence of the formulas above and of the non-degeneracy of the symplectic form: for each $z\in N$,
\begin{multline*}
\Lie^{k}_1(X_{H_1},X_{H_2})(z)=T_z N 
\\
\Leftrightarrow \Span_{T^*_z N} \Big\{ d_z H_1,d_z H_2, d_z \{H_1,H_2\},\dots, d_z \underbrace{\{H_{1},\{H_{1},\{\dots,\{H_{1}}_\text{$k$ times},H_{2}\}\dots\}\}\} \Big\}=T^*_z N.
\end{multline*}
and of course in this case we also have $\Lie^{k}(X_{H_1},X_{H_2})(z)=T_z N 
$.

The part (ii) of the following statement is also known as Krener's Theorem.

\begin{theo}[Local Rashevski-Chow Theorem]
\label{theo:chow-rash local}
Let $X_1,X_2$ be two $C^k$ vector fields \textup{(}$k\ge 1$\textup{)} on the manifold $M$. Let $z\in M$.
\begin{itemize}
\item[(i)] (time-unoriented case) 
If $\Lie^{k-1} (X_1,X_2)(z)=T_z M$, then $z$ is contained in the interior of its orbit $\Oo_{X_1,X_2}(z)$: 
\[
z\in\interior\Oo_{X_1,X_2}(z).
\]
\item[(ii)] (time-oriented case) 
If $\Lie^{k-1}(X_1,X_2)(z)=T_z M$, then $z$ is contained in the closure of the interior of its positive reachable set $\Oo^+_{X_1,X_2}(z)$, as well as in the closure of the interior of its negative reachable set $\Oo^-_{X_1,X_2}(z)$:
\[
z\in\cl\interior\Oo^+_{X_1,X_2}(z)\cap\cl\interior\Oo^-_{X_1,X_2}(z).
\]
\end{itemize}
\end{theo}

We are  now going to state the global counterpart to the previous theorem. In the time-oriented case we shall make the additional assumption that $X_1$ and $X_2$ are complete vector fields and that all points of $M$ are non-wandering for both $X_1$ and $X_2$. Let us recall the precise definition.

\begin{defin}
Let $X$ be a complete $C^1$  vector field  on the manifold $M$ with flow $\{\phi^t_X\}_{t\in\R}$. A point $z\in M$ is said to be non-wandering for $X$ if for every neighborhood $U$ of $z$ and every $t>0$ there exists $t'>t$ such that
\[
\phi^{t'}_X(U)\cap U\neq\emptyset.
\]
\end{defin}

The set of non-wandering points is closed. Note that a point is non-wandering for $X$ if and only if it is non-wandering for $-X$. Note also that if the flow of $X$ preserve a measure of full support and if $M$ is compact then each point is non-wandering for $X$, by the Poincaré Recurrence Theorem. This is the case for an Hamiltonian vector field on a compact symplectic manifold.

\begin{theo}
\label{theo:chow-rash global}
Let $X_1,X_2$ be two $C^k$ vector fields \textup{(}$k\ge 1$\textup{)} on the manifold $M$.
\begin{itemize}
\item[(i)] (time-unoriented case)\quad If 
 every $z\in M$ satisfies
 \[
 z\in\interior \Oo_{X_1,X_2}(z)
 \]
 then
\[
\Oo_{X_1,X_2}(z)=M \qquad\forall\, z\in M.
\]
\item[(ii)] (time-oriented case)\quad Assume that $X_1$ and $X_2$ are complete and that all points of $M$ are non-wandering for both $X_1$ and $X_2$.
 If 
\begin{equation}
\label{eq:clint}
z\in\cl\interior\Oo^+_{X_1,X_2}(z)\cap\cl\interior\Oo^-_{X_1,X_2}(z) \qquad\forall\ z\in M
\end{equation}
then
\[
\Oo^+_{X_1,X_2}(z)
=M\quad\text{ and }\quad \Oo^-_{X_1,X_2}(z)
=M\qquad\forall\, z\in M.
\]
\end{itemize}
\end{theo}

\begin{proof}
The part $(i)$ is an immediate consequence of the connectedness of $M$. The part $(ii)$ is a consequence of the Orbit Theorem (see \cite{AgrSac04} or \cite{Jur97}), and of Corollary 8.1 and Proposition 8.2   in \cite{AgrSac04} as we now quickly recall.
 Indeed, from the Orbit Theorem and the fact that $\Oo_{X_1,X_2}(z)$ has non-empty interior (due to assumption \eqref{eq:clint}), we get that $\Oo_{X_1,X_2}(z)$ is an open subset of $M$ for every $z$. By the part $(i)$ of Theorem \ref{theo:chow-rash global} we deduce
\[
\Oo_{X_1,X_2}(z)=M\qquad\forall\ z\in M.
\]
From Proposition 8.2 in \cite{AgrSac04} and the non-wandering assumption\footnote{In \cite{AgrSac04} the terminology `Poisson stable' rather than `non-wandering' is used.} on $X_1$ and $X_2$ we have
\[
\Oo_{X_1,X_2}(z)\subseteq\cl\Oo^+_{X_1,X_2}(z)\qquad\forall\ z\in M.
\]
Putting together the two above relations yields
\[
M=\cl\Oo^+_{X_1,X_2}(z)\qquad\forall\ z\in M.
\]
The conclusion now follows by Corollary 8.1 in \cite{AgrSac04}. 
\end{proof}

Note that Proposition 8.2 and Corollary 8.1 in \cite{AgrSac04} are therein stated under the assumption that $\Lie(X_1,X_2)(z)=T_z M$ for every $z\in M$, but the proofs hold unchanged under the weaker assumption \eqref{eq:clint}. We will need this slightly more general formulation in the sequel.

\begin{rem}
\label{rem:orbit}
{\rm 
We will use in the sequel the following basic fact: given two $C^1$ vector fields $X_1,X_2$  on $M$  and  $z\in M$, we have
\[
z\in\interior \Oo_{X_1,X_2}(z)\ \Leftrightarrow\ \interior\Oo_{X_1,X_2}(z)\neq\emptyset.
\]
Indeed, let us suppose that the orbit $\Oo_{X_1,X_2}(z)$ has non-empty interior and let us prove that it is open. By changing the point $z$ if necessary, we can suppose that $z$ belongs to the interior of $\Oo_{X_1,X_2}(z)$. Let now $z'$ be another point of $\Oo_{X_1,X_2}(z)$. We want to prove that $z'$ belongs to the interior of $\Oo_{X_1,X_2}(z)$ as well. By definition of orbit, there exists $\phi=\phi^{t_k}_{X_{i_k}}\circ\dots\circ \phi^{t_1}_{X_{i_1}}$
 for some $k\in\N,\, t_1,\dots,t_k\in\R,\, i_1,\dots,i_k\in\{1,2\},$ such that
 \[
 z\in\Dom\phi\quad\text{and}\quad\phi(z)=z'.
 \]
Since $\Dom\phi$ is open and $z\in\interior\Oo_{X_1,X_2}(z)$, the set $\Dom\phi\cap\Oo_{X_1,X_2}(z)$ is a neighborhood of $z$. Since $\phi$ is a local diffeomorphism, the set $\phi(\Dom\phi\cap\Oo_{X_1,X_2}(z))$ is a neighborhood of $z'$ contained in $\Oo_{X_1,X_2}(z)$. This proves that $z'\in\interior\Oo_{X_1,X_2}(z)$.
}
\end{rem}

\subsection{The Thom Transversality Theorem. Rectifiable sets of positive codimension}
\label{sect:thom}

In this section we recall from \cite{Ber10} the notion of rectifiable set of positive codimension in Banach and Frechet spaces. It is quite clear from the definitions below that such a set is automatically Baire-meager (i.e.~it is contained in a countable union of closed sets with empty interior, or, equivalently, its complementary is generic), but the viceversa is not true in general. 
 In this sense, the notion of rectifiable set of positive codimension is a stronger notion of ``smallness''  than the one of having generic complementary. 
 
 Let us first give the definition for Banach spaces. 
  We shall present later the extension to Frechet spaces.

\begin{defin}
The subset $A$ in the Banach space $F$ is a 
\emph{Lipschitz graph of codimension $c$} 
if there exists a splitting $F=E\oplus G$, with $\dim G=c$ and a Lipschitz
map $g:E\longrightarrow G$ such that 
\[
A\subseteq\{ x+g(x), x\in E\}.
\]

A subset $A\subseteq F$ is \emph{rectifiable of codimension $c\in\N$} if it is 
a countable union $A=\cup _n \varphi_n(A_n)$ where 
\begin{itemize}
 \item $\varphi_n:U_n\longrightarrow F$ is a Fredholm map \footnote{A Fredholm map of index $i$ between separable Banach spaces is a $C^1$ map such that the differential is Fredholm of index $i$ at every point (recall that the index is locally constant).}
  of index $i_n$ defined
on an open subset $U_n$ in a separable Banach space $F_n$.
\item $A_n\subseteq U_n$ is a Lipschitz graph of codimension $c+i_n$ in $F_n$.
\end{itemize}

Finally, a subset $A\subseteq F$ is \emph{rectifiable of positive codimension} if it is rectifiable of codimension $c\in\N$ for some $c\ge1$.
\end{defin}

In this paper, the Banach spaces under consideration will mostly be the spaces $C^k(M), k\in\N$, of  real functions of class $C^k$ on a compact manifold $M$.

We shall  occasionally consider the case of compact manifolds with boundary. More precisely, if $M$ is any such manifold, we will consider the space $C^k(M)$ defined as
\begin{multline*}
C^k(M)=\Big\{ f\colon M\to\R:  f \text{ is of class $C^k$ in $M\setminus \partial M$, $f$ is continuous up to the boundary and }
\\
\text{all partial derivatives of $f$ of order $\le k$ extend by continuity to the boundary}  \Big\}
\end{multline*}
The space $C^k(M)$ is a Banach space when it is endowed with the norm
\begin{equation}
\label{eq:Ckboundary}
\left\| f \right\|_{C^k(M)}= \max_{0\le|\alpha|\le k}\sup_{x\in M\setminus\partial M}\left| \partial_\alpha f(x) \right|
\end{equation}
the maximum being taken over all multi-indexes $\alpha=(\alpha_1,\dots,\alpha_m)\in\N^d$ of length $\le k$.  Here $m$ is the dimension of $M$.
For a function $f\in C^k(M)$, we will regard its $k$-jet $j^k f$ as a function from $M\setminus \partial M$ to $J^k(M\setminus\partial M,\R)$.

\medbreak

We will also consider the space $C^k(M)$ when $M$ is  a non-compact manifold without boundary. In this case $C^k(M)$ is no more a Banach space, it is however a Frechet space in the usual way, i.e.\ the Frechet topology is given by the family of seminorms $\left\{\|\cdot\|_{C^k(K_n)}\right\}_{n\in\N}$, where $K_n$ is any increasing sequence of compact sets exhausting $M$ and $\|\cdot\|_{C^k(K_n)}$ is defined as in \eqref{eq:Ckboundary}. In fact it is not restrictive to assume that each $K_n$ is a smooth manifold with boundary.

The definition of rectifiable subset of positive codimension extends to Frechet spaces as follows (cf.\ \cite[Section 3]{Ber10}, where a different terminology was used: we call `rectifiable' here what was called `countably rectifiable' there):

\begin{defin}
\label{def:frechet}
The subset $A$ of the Frechet space $F$ is rectifiable of codimension $c$ if it is a countable union $A=\cup_n A_n$ where each $A_n$ satisfies: there exists a Banach space $B_n$ and a continuous linear map $P_n\colon F\to B_n$ with dense range such that $P_n(A_n)$ is rectifiable of codimension $c$ in $B_n$. 
\end{defin}

Let us point out from \cite[Prop.~16]
{Ber10} the following useful compatibility property among the spaces $C^k(M)$ with respect to different values of $k$.

\begin{prop}
\label{prop:compatibility}
Let $M$ be a  manifold. If $A\subseteq C^k(M)$ is rectifiable of codimension $c$,
 and $k'\geq k$, then $A\cap C^{k'}(M)$
 is rectifiable of codimension $c$ in $C^{k'}(M)$.
\end{prop}

The next theorem can be seen as a finer version of the ``avoiding case'' of the classical Thom transversality Theorem in jet spaces. 

\begin{theo}
\label{theo:thom}
Let $M$ be a manifold, and $W\subseteq J^k(M,\R)$ a $C^1$-submanifold such that
\[
\codim W\ge \dim M+1.
\]
Then, for every $r\ge 1$ the set
\[
\left\{f\in C^{k+r}(M):j^k f(M)\cap W\neq \emptyset  \right\}
\]
is rectifiable of codimension equal to $\codim W-\dim M\ge1$ in  $C^{k+r}(M)$.
\end{theo}

\begin{proof}
Let $B=\{z\in\R^d:|z|<1\}$ be the open unit disc in $\R^d$, and let $W$ be a $C^1$ submanifold of $J^k(B,\R)$ with $\codim W\ge d+1$. In  \cite[Prop.\ 7]{BerMan12} it is proved that  for every $r\ge1$ the set 
\[
\left\{f\in C^{k+r}(\bar B):j^k f(B)\cap W\neq\emptyset\right\}
\]
is rectifiable of codimension equal to $\codim W-d$ in the Banach space $C^{k+r}(\bar B)$   defined according to  \eqref{eq:Ckboundary}.

The proof given in \cite{BerMan12}  actually yields, with essentially no modifications, the following more general result: if $M$ is a compact manifold with or without boundary, and $W$ is a $C^1$ submanifold of $J^k(M\setminus\partial M,\R)$ with $\codim W\ge\dim M+1$, then for every $r\ge 1$ the set
\begin{equation}
\label{eq:withboundary}
\left\{f\in C^{k+r}(M):j^k f(M\setminus\partial M)\cap W\neq\emptyset\right\}
\end{equation}
is rectifiable of codimension $\codim W-\dim M$ in the Banach space $C^k(M)$.  This proves the present statement in the case of compact $M$.

In order to end the proof of the theorem it remains to consider the non-compact case. Let $M$ be a non-compact manifold (without boundary), and call $A\subseteq C^{k+r}(M)$ the subset in the statement, namely
\[
A=\left\{f\in C^{k+r}(M):j^k f(M)\cap W\neq \emptyset  \right\}.
\]
Let $K_n\subseteq M, n\in\N$, be a sequence of smooth compact sets exhausting $M$. Each of them is a smooth manifold with boundary. Call $P_n\colon C^k(M)\to C^k(K_n)$ the natural projection. Note that the space $J^k(K_n\setminus \partial K_n,\R)$ is naturally included in $J^k(M,\R)$, and that $A=\cup_n A_n$ where
\[
A_n=P_n^{-1}\Big(\left\{ f\in C^{k+r}(K_n):j^k f(K_n\setminus\partial K_n)\cap W\neq \emptyset    \right\}\Big).
\]
For each $n$ the set
\[
\left\{ f\in C^{k+r}(K_n):j^k f(K_n\setminus\partial K_n)\cap W\neq \emptyset    \right\}
\]
is rectifiable of codimension $\codim W-\dim M$ in $C^{k+r}(K_n)$, because we are in the same situation of the set in \eqref{eq:withboundary}. Moreover, each $P_n$ is a surjective (cf.\ \cite{Hes41,See64,Bie80} for more general statements) continuous linear operator from the Frechet space $C^k(M)$ to the Banach space $C^k(K_n)$   and,  \emph{a fortiori}, it has dense range. Hence, by Definition \ref{def:frechet}, the set $A$ is rectifiable of codimension $\codim W-\dim M$ in the Frechet space $C^k(M)$, as desired.
\end{proof}

Clearly, the above theorem still holds true if one replaces $W$ by a countable union of $C^1$ submanifolds of codimension greater or equal than $\dim M+1$. In fact it is still true when $W$ is a rectifiable set of codimension greater or equal than $\dim M+1$, this also follows from \cite[Prop.\ 7]{BerMan12}.

\subsection{The Hamiltonian Flow-Box Theorem}
\label{sect:flowbox}

\begin{theo}
\label{theo:ham flowbox}
Let $(N,\omega)$ be a symplectic $2d$-dimensional manifold, and $H\colon N\to \R$ be a $C^k$ function, $k\ge 2$. Let $z\in N$ be such that $d_{z} H\neq 0$. Then, there exist a $C^{k-1}$ chart $\psi\colon U\to\R^{2d}$ defined in a neighborhood $U$ of $z$ such that  
\[
H\circ \psi^{-1}(x,p)=p_1\qquad\text{and}\qquad (\psi^{-1})^* \omega_0=\omega
\]
where $\omega_0$ is the standard symplectic form on $\R^{2d}$ and $(x,p)=(x_1,\dots,x_d,p_1,\dots,p_d)$ are the associated Darboux coordinates.
\end{theo}

\begin{proof}
We refer to \cite[Theorem 5.2.19]{AbrMar78}, where in fact the function $H$ is assumed to be smooth. However, the diffeomorphism $\psi$ is therein constructed using the flow of $H$, which is $C^{k-1}$ if $H$ is $C^k$.
\end{proof}

\section{The autonomous case}
\label{sect:autonomous}

In this section we prove the main results presented in the introduction. For the convenience of the reader, we restate them as Theorem \ref{theo:tu-aut} and Theorem \ref{theo:to-aut} below.

\begin{theo}[Time-unoriented, autonomous case]
\label{theo:tu-aut}
Let $N$ be a symplectic manifold of dimension $2d$, and let $H_1\in C^{4d+1}(N)$ satisfy the non-degeneracy assumption $\eqref{ass:H_1}$. Then, for every $k\ge 4d$, the set 
\[
\left\{ H_2\in C^{k}(N): \Oo_{H_1,H_2}(z)=N\ \forall\,z\in N \right\}
\]
has rectifiable complementary of codimension $\ge 1$ in $C^{k}(N)$. 
\end{theo}

\begin{theo}[Time-oriented, autonomous case]
\label{theo:to-aut}
Let $N$ be a compact symplectic manifold of dimension $2d$, and let $H_1\in C^{4d+1}(N)$ satisfy the non-degeneracy assumption $\eqref{ass:H_1}$. Then, for every $k\ge 4d$, the set
\[
\left\{ H_2\in C^k(N): \mathcal O^+_{H_1,H_2}(z)=N\ \ \forall\,z\in N \right\} 
\]
has rectifiable complementary of codimension $\ge 1$ in $C^{k}(N)$.
\end{theo}

The proof of both theorems makes use of the following two results.

\begin{lemma}
\label{lemma:codim}
Let $N$ be a  symplectic manifold, and let $H_1$ be of class $C^{k+1}$ ($k\ge 2$) and satisfy the non-degeneracy Assumption \eqref{ass:H_1}. Let us set
\[
N_0=\{z\in N:d_z H_1=0\},\qquad N'=N\setminus N_0
\]
and define subsets $W',W_0\subseteq J^{k-1}(N,\R)$ by
\begin{align*}
W'&{=}\left\{ j \in J^{k-1}(N,\R): \text{if }j=j^{k-1}_zH_2\text{ then } z\in N' \text{ and } \Lie^{k-2}_1(X_{H_1},X_{H_2})(z)\subsetneq T_z N \right\}
\\
W_0 &{=}\left\{ j \in J^{k-1}(N,\R): \text{if }j=j^{k-1}_zH_2\text{ then } z\in N_0 \text{ and }  X_{H_2}(z)\in T N_0 \right\}.\footnotemark
\end{align*}
\footnotetext{Recall that in Assumption \ref{ass:H_1} we defined $T N_0$ as the union $\bigcup_{l\in\N}T N^{l}$ where $\{N^{l}\}_l$ is a once-for-all fixed countable family of submanifolds of codimension greater or equal than $\codim N_0$ and whose union covers $N_0$.}
We have:
\begin{align*}
\codim_{J^{k-1}(N,\R)} W'&= k-\dim N+1 
\\
\codim_{J^{k-1}(N,\R)} W_0 &\ge \dim N+2.
\end{align*}
In particular, 
\[
\codim (W'\cup W_0)\ge \dim N+1\quad\text{ as soon as }\quad k\ge 2\dim N.
\]
\end{lemma}

\begin{prop}
\label{prop:get out}
Let $H_1, N_0,N'$ be as in Lemma \ref{lemma:codim} and $H_2\in C^k(N)$, $k\ge2$, satisfy
\[
X_{H_2}(z)\notin TN_0\qquad\forall\,z\in N_0.
\]
Then, the set
\[
\left\{t\in\R:\phi^t_{H_2}(z)\in N_0\right\}
\]
has empty interior for every $z\in N_0$. In particular, each $z\in N_0$ is accumulated by points in $\Oo^+_{H_1,H_2}(z)\cap N'$ as well as by points in $\Oo^-_{H_1,H_2}(z)\cap N'$.
\end{prop}

\begin{proof}[{\bf Proof of Theorem \ref{theo:tu-aut}}]
First of all, let us notice that it suffices to prove the result for $k=4d$, thanks to Proposition \ref{prop:compatibility}.

Let $W'$ and $W_0$ be as in Lemma \ref{lemma:codim}. By that Lemma and by the positive-codimension version of the Thom transversality theorem (Theorem \ref{theo:thom}) we get that the set
\[
\left\{ H_2\in C^{4d}(N):j^{4d-1} H_2(N)\cap \big(W'\cup W_0\big)=\emptyset \right\}
\]
has rectifiable complementary of codimension $\ge 1$ in $C^{4d}(N)$. Thus it suffices to prove that any $H_2$ belonging to the set above satisfies $\Oo_{H_1,H_2}(z)=N$ for all $z\in N$. 

Let $H_2$ belong to the set above. If the set $N_0=\{z\in N:d_z H_1=0\}$ is empty, then, recalling the definition of $W'$, the conclusion  follows from Theorems \ref{theo:chow-rash local} and \ref{theo:chow-rash global}. If $N_0$ is not empty, the argument can be adapted as follows.

By connectedness of $N$, it suffices to prove that the orbit  of  every point $z\in N$ is open. 
 By Remark \ref{rem:orbit}, an orbit either is open or it has empty interior. 
 Hence we are reduced to prove that the orbit of every point  $z\in N$ has non-empty interior. This is true if $z\in N'=N\setminus N_0$, by definition of $W'$ and by the Rashevski-Chow Theorem. This is equally true if $z\in N_0$, because in this case we get by Proposition \ref{prop:get out} that the orbit of $z$ intersects (hence coincide with) the orbit of some point in $N'$, and we just proved that such an orbit has non-empty interior.

\end{proof}

\begin{proof}[{\bf Proof of Theorem \ref{theo:to-aut}}]
The proof is analogous to the one of Theorem \ref{theo:tu-aut} above. As before, it suffices to prove the result for $k=4d$. By repeating the first part of that proof, we  get that  the set
\begin{equation}
\label{eq:good set}
\left\{ H_2\in C^{4d}(N):j^{4d-1} H_2(N)\cap \big(W'\cup W_0\big)=\emptyset \right\}
\end{equation}
has rectifiable complementary of codimension $\ge 1$ in $C^{4d}(N)$. Hence it suffices to prove that any $H_2$ belonging to the set above satisfies the properties stated in the theorem.

Let $H_2$ belong to the set above.  If the set $N_0=\{z\in N:d_z H_1=0\}$ is empty, then the conclusion immediately follows from the time-oriented version of the  Rashevski-Chow Theorem (see Theorems \ref{theo:chow-rash local} and \ref{theo:chow-rash global})). Note that any flow is complete on the compact manifold $N$, and any Hamiltonian flow on $N$ has the property that all points are non-wandering by Poincaré recurrence Theorem, thus the Theorem \ref{theo:chow-rash global}(ii) is indeed applicable.

If $N_0$ is not empty we adapt the argument as follows: by the Theorem \ref{theo:chow-rash global}(ii) it suffices to prove that every $z\in N$ satisfies
\[
z\in\cl\interior\Oo^+_{X_1,X_2}(z)\cap\cl\interior\Oo^-_{X_1,X_2}(z).
\]
This is true if $z\in N'=N\setminus N_0$ by  Krener's Theorem, see Thm.\ \ref{theo:chow-rash local}(ii). 
 This is equally true if $z\in N_0$ as we  now show. Indeed, in this case we get by Proposition \ref{prop:get out} that $z$ is accumulated by a sequence $(z_n)_{n\in\N}$ of points in $\Oo^+_{H_1,H_2}(z)\cap N'$. Since each $z_n$ belongs to $N'$, we know by the previous case that $z_n$ is accumulated by a sequence $(z_{n,k})_{k\in\N}$ of points in $\interior\Oo^+_{H_1,H_2}(z_n)$. By definition of reachable set, we have $\interior\Oo^+_{H_1,H_2}(z_n)\subseteq \interior\Oo^+_{H_1,H_2}(z)$ for each $n$. Hence the subset $\{z_{n,k}\}_{n,k\in\N}$ is contained in $\interior\Oo^+_{H_1,H_2}(z)$ and has $z$ as  an accumulation point. This proves that $z\in\cl\interior\Oo^+_{H_1,H_2}(z)$. The proof for $\Oo^-$ is analogous.
\end{proof}

\begin{proof}[{\bf Proof of Lemma \ref{lemma:codim}}]
Let us first prove the inequality about $W'$. If $k< 2d$ then the condition $\Lie^{k-2}_1(X_{H_1},X_{H_2})\subsetneq T_z N$ is trivially satisfied because $\Lie^{k-2}_1(X_{H_1},X_{H_2})$ is spanned by $k$ vector fields, and the conclusion is true. Let us then suppose $k\ge 2d$.

The set $W'$ is locally defined (above the open set $N'\subseteq N$) by 
 the inequality
\begin{equation}
\label{eq:def W'}
\rk\begin{pmatrix}
d_z H_1
\\
d_z H_2
\\
d_z \{H_1,H_2\}
\\
\vdots
\\
d_z \underbrace{\{H_{1},\{H_{1},\{\dots,\{H_{1}}_\text{$k-2$ times},H_{2}\}\dots\}\}\}
\end{pmatrix}<2d.
\end{equation}
By the Hamiltonian Flow-box Theorem (Thm.\ \ref{theo:ham flowbox}) we can find, near any arbitrary point of $N'$, a local $C^{k}$ symplectic chart $\psi$ yielding identification with Darboux coordinates $z=(x,p)=(x_1,\dots,x_d,p_1,\dots,p_d)$ such that $H_1(z)=p_1$.\footnote{Note that $\psi$ also induces a change of coordinates on $J^{k-1}(N,\R)$, which is of class $C^1$ because $\psi$ is of class $C^{k}$. Hence the codimension of $W'$ is the same as the codimension of its image under this diffeomorphism. This legitimates the subsequent computations (and accounts for the requirement $H_1\in C^{k+1}$ rather than just $C^k$).
}
 A computation then shows that, in these coordinates, 
 \[
\underbrace{\{H_{1},\{H_{1},\{\dots,\{H_{1}}_\text{$m$ times},H_{2}\}\dots\}\}\}(z)=\partial_{x_1^m} H_2(z),
\]
for any function $H_2$ differentiable enough. As a consequence,
\begin{multline*}
d_z\ \underbrace{\{H_{1},\{H_{1},\{\dots,\{H_{1}}_\text{$m$ times},H_{2}\}\dots\}\}\}
=\left(\partial_{x_1^{m+1}}H_2\;,\;\dots\;,\;\partial_{x_d x_1^{m}}H_2\;,\;\partial_{p_1 x_1^{m}}H_2\;,\;\dots\;,\;\partial_{p_d x_1^{m}}H_2\right)(z)
\end{multline*}
and the definition \eqref{eq:def W'} for $W'$ becomes more explicit:
\[
{\rm rank\ }\begin{pmatrix}
0 & \dots & 0 & 1 & 0 & \dots & 0
\\
\partial_{x_1}H_2  & \dots &\partial_{x_d}H_2 & \partial_{p_1}H_2 & \partial_{p_2}H_2 & \dots & \partial_{p_d }H_2
\\
\vdots & \ddots &\vdots&\vdots&\vdots&\ddots&\vdots
\\
\partial_{x_1^{k-1}}H_2 & \dots &\partial_{x_d x_1^{k-2}}H_2 & \partial_{p_1 x_1^{k-2}}H_2 & \partial_{p_2 x_1^{k-2}}H_2 & \dots & \partial_{p_d x_1^{k-2}}H_2
\end{pmatrix}(z) < 2d.
\]
This is a $k\times 2d$ -- matrix. The first row corresponds to $d_z H_1$. The other rows correspond to the iterated Lie brackets computed above up to $m=k-2$, and their entries are clearly independent when regarded as jet-coordinates. We deduce that the codimension of the set $W'$ is the same as the codimension of the set of $(k-1)\times (2d-1)$ matrices with non-maximal rank. Since we are assuming $k\ge 2d$, this codimension is well-known to be
\[
k-2d+1
\]
as desired.
\medbreak

Let us now prove the inequality $\codim_{J^{k-1}(N,\R)} W_0\ge 2d+2$. By the Assumption \ref{ass:H_1} we have
\[
N_0\subseteq\bigcup_{l\in\N} N^{l}
\]
where each $N^{l}$ is a $C^2$ submanifold of $N$ of codimension greater or equal than $d+1$. Each tangent space $TN^{l}$ is a $C^1$ submanifold of $TN$ and
\[
\codim_{TN}TN^{l}\ge 2(d+1).
\]
Since the map
\begin{align*}
J^{k-1}(N,\R)&\to TN
\\
j^{k-1}_z H_2&\mapsto X_{H_2}(z)
\end{align*}
is a submersion for all $k\ge 2$ (due to non-degeneracy of the symplectic form), and $W_0$ is precisely the preimage of $TN_0=\cup_{l\in\N} TN^{l}$ via this map, we get
\[
\codim_{J^{k-1}}W_0=\codim_{TN}TN_0\ge 2(d+1),
\]
as desired.
\end{proof}

\begin{proof}[{\bf Proof of Proposition \ref{prop:get out}}]
Recall that $N_0$ is contained in a countable union of submanifolds $\{N^l\}_l$ and, by definition, $TN_0=\cup_l TN^l$. We want to prove that for each $z\in N_0$ the set
\[
\{t\in\R:\phi^t_{H_2}(z)\in N_0\}
\]
has empty interior. By Baire's Theorem, it suffices to prove that, for each fixed $l$, the closure of the set
\[
\{t\in\R:\phi^t_{H_2}(z)\in N^l\}
\]
has empty interior. This is easily seen to be true: the assumption
\[
X_{H_2}(z)\notin T_z N^l\qquad\forall\,z\in N^l
\]
implies that the set above is constituted by isolated points, and the closure of such a set always has empty interior.
\end{proof}

\section{The non-autonomous case}
\label{sect:per}

In this section we extend to the time-dependent case the results obtained in the previous section.

Let $N$ be a symplectic manifold of dimension $2d$, and $H_1,H_2\colon N\times\R\to\R$ two time-dependent Hamiltonians. They give rise to two time-dependent Hamiltonian vector fields on $N$ denoted respectively by  $X_{H_1},X_{H_2}$. Equivalently, they give rise to two time-independent vector fields $\tilde X_{H_1}$ and $\tilde X_{H_2}$ on $N\times\R$ defined by 
\[
\tilde X_{H_i}(z,t)=(X_{H_i}(z,t),1)\ \in\ T_{(z,t)} (N\times \R)\cong T_z N\times\R,\qquad (z,t)\in N\times\R,\ i\in\{1,2\}.
\]

We shall occasionally denote by $M$ the manifold $N\times\R$. If $(z,t)\in N\times\R$, in this section its orbit $\Oo_{H_1,H_2}(z,t)$ has to be intended as the orbit of $(z,t)$ through $\tilde X_{H_1}$ and $\tilde X_{H_2}$. It is therefore a subset of $N\times\R$. We adopt the analogous definition  for the reachable set $\Oo^+_{H_1,H_2}(z,t)$. 

We will make the following assumption on $H_1$:

\begin{ass}
\label{ass:H_1per}
The Hamiltonian $H_1$ is continuously differentiable and the subset $M_0\subseteq N\times\R$ defined by
\[
M_0=\{ (z,t)\in N\times\R:d_z H_1(z,t)=0 \},
\]
is contained in a countable union $\cup_{l\in\N} M^{l}$ of ${C^2}$-submanifolds of $N\times\R$ of codimension greater or equal than $
\frac{\dim N}2 +1$. Given such a family  $\{M^l\}_l$ we define the tangent space $TM_0$ as
\[
TM_0=\bigcup_{l\in \N} TM^{l}.
\]
Note that each $TM^l$ is a submanifold of class $C^1$ in $T(N\times\R)$. According to Definition \ref{defin:codim} we have
\begin{align*}
\codim_{N\times\R} M_0&\ge \frac{\dim N}2 +1
\\
\codim_{T(N\times\R)} TM_0&\ge 2\, \codim_{N\times\R} M_0\ge \dim N+2.
\end{align*}
\end{ass}

As in the autonomous case, the above assumption is generic by an easy application of the classical Thom transversality Theorem.

Let us now state the time-dependent versions of Theorems \ref{theo:tu-aut} and \ref{theo:to-aut}.

\begin{theo}[Time-unoriented, non-autonomous case]
\label{theo:tu-per}
Let $N$ be a symplectic manifold of dimension $2d$, and let $H_1\in C^{4d+2}(N\times\R)$ satisfy the non-degeneracy assumption $\eqref{ass:H_1per}$. Then, for every $k\ge 4d+1$ the set
\[
\left\{ H_2\in C^{k}(N\times\R): \Oo_{H_1,H_2}(z,t)=N\times\R\ \  \forall\,(z,t)\in N \right\}
\]
has rectifiable complementary of codimension $\ge 1$ in $C^k(N\times\R)$. 
\end{theo}

Since both vector fields $\tilde X_{H_1}$ and $\tilde X_{H_2}$ induce the equation $\dot t=1$ on the $t$-variable, it is obviously impossible for the reachable set of a point to be the whole of $N\times\R$. For this reason we make the assumption that $H_1$ and $H_2$ are one-periodic in time, i.e.\ they are defined on $N\times\T$ where $\T=\R/\Z$. We then regard the reachable set of a point $(z,t)$ as a subset of $N\times\T$.

\begin{theo}[Time-oriented, periodic case]
\label{theo:to-per}
Let $N$ be a compact symplectic manifold of dimension $2d$, and let $H_1\in C^{4d+2}(N\times\T)$ satisfy the non-degeneracy Assumption $\eqref{ass:H_1per}$.   For every $k\ge 4d+1$  the set 
\begin{equation*}
\Big\{ H_2\in C^k(N\times\T): \Oo^+_{H_1,H_2}(z,t)=N\times\T\ \ \forall\ (z,t)\in N\times\T\Big\}
\end{equation*}
has rectifiable complementary of codimension $\ge 1$ in  $C^k(N\times \T)$.
\end{theo}

\begin{proof}[{\bf Proof of Theorems \ref{theo:tu-per} and \ref{theo:to-per}}]
The proof is the same as in the autonomous case (Theorems \ref{theo:tu-aut} and \ref{theo:to-aut}), once the time-dependent counterpart to Lemma \ref{lemma:codim} has been established.  This is the content of Lemma \ref{lemma:codim2} below. Note that for the time-oriented part one has to apply  at some point the Theorem \ref{theo:chow-rash global} (ii) which makes the assumption that the two considered flows have non-wandering dynamics. Since $\tilde X_{H_1}$ and $\tilde X_{H_2}$ preserve the measure $\mu\oplus dt$ on the compact manifold $N\times\T$, this assumption is indeed fulfilled by the Poincaré Recurrence Theorem.
\end{proof}

\begin{lemma}
\label{lemma:codim2}
Let $N$ be a symplectic manifold of dimension $2d$, and let $H_1\colon N\times\R$ be of class $C^{k+1}$ ($k\ge 2$) and satisfy the non-degeneracy assumption \eqref{ass:H_1per}. Let us define the subsets $M_0,M'\subseteq N\times\R$ by
\[
M_0=\{(z,t)\in N\times\R:d_z H_1(z,t)=0\},\qquad M'=(N\times\R) \setminus M_0
\]
and the subsets $W',W_0\subseteq J^{k-1}(N\times\R,\R)$ by
\begin{align*}
W'&{=}\left\{ j \in J^{k-1}(N\times\R,\R): \text{if }j=j^{k-1}_{(z,t)}H_2\text{ then } (z,t)\in M' \text{ and } \Lie^{k-2}_1(\tilde X_{H_1},\tilde X_{H_2})(z,t)\subsetneq T_{(z,t)} (N\times\R) \right\}
\\
W_0 &{=}\left\{ j \in J^{k-1}(N\times\R,\R): \text{if }j=j^{k-1}_{(z,t)}H_2\text{ then } (z,t)\in M_0 \text{ and }  \tilde X_{H_2}(z)\in T M_0 \right\}.
\end{align*}
We have:
\begin{align*}
\codim_{J^{k-1}(N\times\R,\R)} W'&= k-\dim N 
\\
\codim_{J^{k-1}(N\times\R,\R)} W_0 &\ge \dim (N\times\R)+1.
\end{align*}
In particular, $\codim (W'\cup W_0)\ge \dim N+1$ as soon as $k\ge 2\dim N+1$.
\end{lemma}

\begin{rem}
 The lemma above is stated in the non-periodic setting. It holds in the periodic case as well by replacing the jet space $J^{k-1}(N\times\R,\R)$ with  $J^{k-1}(N\times\T,\R)$. The proof remains essentially unchanged.
\end{rem}

\begin{proof}[{\bf Proof of Lemma \ref{lemma:codim2}}]
The proof is very similar to the one of Lemma \ref{lemma:codim}. We start with proving the inequality about $W'$, for which we only consider the case $k\ge 2d+1$. In the other cases  the result is trivial by a dimensional argument.

Let $j_0$ be an arbitrary element of $W'$.
 We are going to prove that in a neighborhood of $j_0$ the codimension of $W'$ is bounded below by the desired value $k-\dim N$. The global bound on the codimension of $W'$ will then follow by standard arguments.  The subset $W'$ is defined by the inequality
\begin{equation}
\label{eq:rank}
\rk
\begin{pmatrix}
\tilde X_{H_1}
\\
\tilde X_{H_2}
\\
[\tilde X_{H_1},\tilde X_{H_2}]
\\
\vdots
\\
\underbrace{[\tilde X_{H_1},[\dots,[\tilde X_{H_1}}_\text{$k-2$ times},\tilde X_{H_2}]\dots]]
\end{pmatrix}
(z,t)<2d+1
\end{equation}
which depends just on the $(k-1)$-jet of $H_2$. Note that, since $H_1$ is of class $C^{k+1}$ and the matrix involves up to $k-1$ derivatives of $H_1$, its entries  are of class $C^2$ in the variable $j=j^{k-1}_{(z,t)}H_2\in J^{k-1}(N\times\R,\R)$.

Let us denote $(z_0,t_0)$ the source of the jet $j_0$. Since $j_0\in W'$, we have $d_z H_1(z_0,t_0)\neq 0$.  By the Hamiltonian Flowbox Theorem (Thm.~\ref{theo:ham flowbox}) applied at $H_1(\cdot,t_0)$, there exists a local $C^k$ diffeomorphism $\psi\times\id$ yielding identification with coordinates $(z,t)=(x,p,t)$ on $U\times\R$,\footnote{Note that $\psi\times\id$ also induces a local diffeomorphism of class $C^1$ from (a subset of) $J^{k-1}(N\times\R,\R)$ to $J^{k-1}(U\times\R,\R)$, thus preserving codimensions of subsets.} with $(x,p)=(x_1,\dots,x_d,p_1,\dots,p_d)$ being standard Darboux variables on the open set $U\subseteq\R^{2d}$, and such that the Hamiltonian in the new coordinates satisfies at time $t_0$
\[
H_1(z,t_0)=p_1\qquad\forall\,z\in U.
\]
Using these coordinates it is easy to compute the restriction at $t=t_0$ of the iterated Lie brackets with $\tilde X_{H_1}$. Indeed, for any $m\ge 1$ and any function $H_2(z,t)$ differentiable enough, the computation (which we omit) yields
\[
\underbrace{[\tilde X_{H_1},[\dots,[\tilde X_{H_1}}_\text{$m$ times},\tilde X_{H_2}]\dots]](z,t_0)=\big(X_{K_m}(z,t_0),0\big)\qquad\forall\,z=(x,p)\in U
\]
where $X_{K_m}$ is the Hamiltonian vector field on $U$ associated to the function $K_m\colon U\to\R$ defined by
\[
z\mapsto K_m(z)=(\partial_{x_1}-\partial_t)^mH_2(z,t_0)+(\partial_{x_1}-\partial_t)^{m-1}\partial_t H_1(z,t_0).
\]
Note that $K_m$ is a sum of two terms, with $H_2$ appearing just in the first summand and $H_1$ just in the second.

 Let us denote $j=(z,t,j')$ the elements of $J^{k-1}(U\times\R,\R)$, where the variable $j'$ regroups all  variables other than the source $(z,t)$, i.e.~$j'$ regroups the value of $H_2$ as well as all its partial derivatives up to order $k-1$.
 By using the explicit computation above we get that in the new coordinates the inequality \eqref{eq:rank} defining $W'$ becomes
\[
\rk
\Big(
A(j')+B_{H_1}(z,t_0)+C_{H_1}(z,t,j')
\Big)
<2d+1
\]
where: 
\begin{itemize}
\item the matrix $C_{H_1}(z,t,j')$ is a $C^1$ matrix which is identically zero for $t=t_0$ and its exact expression does not play any role here; 
\item
 the matrix $A(j')$ is given by\footnote{For reasons of page layout, the matrix $A(j')$ written here is rather the transpose of what it should be in according to the matrix in  \eqref{eq:rank}.}
\[
A(j')=
\begin{pmatrix}
   1  &     \partial_{p_1}H_2       &     \partial_{p_1}\partial_{(x_1-t)}H_2   &  \dots   &   \partial_{p_1}\partial_{(x_1-t)^{k-2}}H_2 
\\
  0 &    \partial_{p_2}H_2       &     \partial_{p_2}\partial_{(x_1-t)}H_2  &  \dots   &   \partial_{p_2}\partial_{(x_1-t)^{k-2}}H_2  
\\
\vdots & \vdots & \vdots & \ddots & \vdots 
\\
 0 &   \partial_{p_d}H_2       &     \partial_{p_d}\partial_{(x_1-t)}H_2   &  \dots   &   \partial_{p_d}\partial_{(x_1-t)^{k-2}}H_2  
\\
 0 &    -\partial_{x_1}H_2       &     -\partial_{x_1}\partial_{(x_1-t)}H_2 &  \dots   &   -\partial_{x_1}\partial_{(x_1-t)^{k-2}}H_2 
\\
\vdots & \vdots & \vdots & \ddots & \vdots 
\\
 0 &    -\partial_{x_d}H_2       &     -\partial_{x_d}\partial_{(x_1-t)}H_2  &  \dots   &   -\partial_{x_d}\partial_{(x_1-t)^{k-2}}H_2
 \\
 1 & 1 & 0 & \dots & 0
\end{pmatrix}
\]
(here the first two columns correspond to $\tilde X_{H_1}$ and  $\tilde X_{H_2}$, and the remaining columns account for the $H_2$-summand in the definition of $K_m$, for $m$ ranging from $1$ to $k-2$); 
\item the $C^1$ matrix $B_{H_1}(z,t_0)$, depending just on $z$ but not on $t$ and $j'$,  accounts for the $H_1$-summand in the definition of $K_m$ and its exact expression does not play any role.
\end{itemize}
From the expressions above it is clear that for $t=t_0$ the variable $j'$ appears just in the matrix $A$. Let us denote by $J^{k-1}_{(z_0,t_0)}$ the fiber of $J^{k-1}(U\times\R,\R)$ over $(z_0,t_0)$. From the expression of $A(j')$ above we deduce that the codimension of 
 $W'\cap J^{k-1}_{(z_0,t_0)}$ in $J^{k-1}_{(z_0,t_0)}$ is the same as 
the codimension of the set of $2d\times (k-1)$ matrices with non-maximal rank. Since we are assuming $k\ge 2d+1$, this codimension is well-known to be
\[
k-2d.
\]
Since $B_{H_1}$ and $C_{H_1}$ are of class $C^1$, this fiberwise estimate allows to deduce the local bound
\[
\codim_{J^{k-1}(U\times\R,\R)} W'\ge k-2d\qquad\text{in a neighborhood of $j_0$},
\]
which is what we wanted to prove.

\medbreak

Let us now prove the inequality about $W_0$. The set $W_0$ is the preimage of $TM_0$ under the map 
\begin{align*}
J^{k-1}(N\times\R,\R)&\to T(N\times\R)\cong TN\times \R\times\R
\\
j^{k-1}_{(z,t)} H_2&\mapsto \big(X_{H_2}(z),t,1\big)
\end{align*}
Let us write $TM_0=\cup_{l\in \N}TM^{l}$ with each $M^{l}$ being a $C^2$ submanifold. It is not difficult to check that the map above is transverse to the $C^1$ submanifold $TM^{l}$ for each $l$. Recalling the Assumption \ref{ass:H_1per} on $H_1$,
 we then deduce
 \[
\codim_{J^{k-1}(N\times\R,\R)}TM_0=\codim_{T(N\times\R)}\bigcup_{l\in\N}TM^{l}\ge 2d+2.\qedhere
\]
\end{proof}

\bibliographystyle{plain}



\end{document}